\documentclass[a4paper, 11pt]{article}

\usepackage{amsmath}
\usepackage{amsfonts}
\usepackage{amssymb}
\usepackage[english]{babel}
\usepackage{graphicx}
\usepackage{amsthm}
\usepackage{mathrsfs}
\usepackage{pgf,tikz}
\usepackage{amstext} 
\usepackage{array}   
\newcolumntype{C}{>{$}c<{$}} 
\definecolor{uququq}{rgb}{0.25,0.25,0.25}
\newtheorem{theorem}{Theorem}[section]
\newtheorem{corollary}[theorem]{Corollary}
\newtheorem{lemma}[theorem]{Lemma}
\newtheorem{proposition}[theorem]{Proposition}

\usetikzlibrary{arrows}
\theoremstyle{definition}
\newtheorem{definition}[theorem]{Definition}
\theoremstyle{definition}
\newtheorem{remark}[theorem]{Remark}
\theoremstyle{definition}
\newtheorem{example}[theorem]{Example}

\newcommand{\Z}{\mathbb{Z}}

\newcommand{\F}{\mathbb{F}}

\newcommand{\comment}[1]{}
\numberwithin{equation}{section}

\begin{document}
\date{}
\title{Fano Kaleidoscopes and their generalizations}
\author{Marco Buratti \thanks{Dipartimento di Matematica e Informatica, Universit\`a di Perugia, via Vanvitelli 1 email: buratti@dmi.unipg.it}\\
\\
Francesca Merola
\thanks{Dipartimento di Matematica e Fisica, Universit\`a Roma Tre, Largo S.L. Murialdo, 1 
email: merola@mat.uniroma3.it }}
\maketitle

\begin{abstract}
In this work we introduce Fano Kaleidoscopes, Hesse Kaleidoscopes and their generalizations. These are a particular kind of colored designs for which we will discuss  general theory, present some constructions and prove existence results. In particular, using difference methods we show  the existence of both a Fano and a Hesse Kaleidoscope on $v$ points when $v$ is a prime or prime power congruent to 1$\pmod{6}$, $v\ne13$. In the Fano case this, together with known results on pairwise balanced designs, allows us to prove the existence of  Kaleidoscopes of order $v$ for many other values of $v$; we discuss what the situation is, on the other hand, in the Hesse and general case.   

\end{abstract}

\section{Introduction}
In this work we introduce Fano Kaleidoscopes and their generalizations, discuss their properties and prove some existence results. A Fano Kaleidoscope of order $v$, briefly a FK$(v)$, is a particular kind of colored design described informally in the following way.

Consider a set ${\cal F}$ of Fano planes whose points belong  to a
 given $v$-set, say $\cal V$, with the seven lines of each plane colored with seven different colors,
 say $c_0,c_1,
 \dots, c_6$.  We say that $\cal F$ is a {\it Fano-Kaleidoscope} of order $v$ (briefly FK$(v)$) if for any
two distinct points $x$, $y$ of $\cal V$ and any color $c_i\in\{c_0,c_1,\dots,c_6\}$ there is exactly one Fano plane of ${\cal F}$ 
whose $c_i$-colored line contains $x$ and $y$.  

A Fano Kaleidoscope can be thought of in the context of colored designs or colored graph decompositions  (see for instance \cite{ABJ,CRYI,CRYII,CRYIII,LW}).
The notion of colored design turns out to be quite powerful, and encompasses many design-theoretic structures, for instance perfect cycle systems, whist tournaments, nested cycle systems; many other examples are contained in \cite{LW}. 

The most important result in this area is Lamken and Wilson's powerful asymptotic existence result \cite{LW},  that  show asymptotic existence of edge-colored graph decompositions of $K_v$ under some reasonable necessary conditions. No bound on the asymptotics in $v$ is given in \cite{LW}.

The notion of Fano Kaleidoscope may be generalized for example by considering, in place of the set $\cal F$ of Fano planes, a set $\cal H$ of 2-$(9,3,1)$-designs, i.e. affine planes $AG(2,3)$ (or Hesse planes), whose points belong  to a
 given $v$-set $\cal V$. In this case the 12 lines of each plane will be colored with 12assigned  different colors, and we require that for any choice of two points of $\cal V$ and of a color $c$ there is only one plane in which the line through the two chosen points has color $c$; in this situation we will speak of a Hesse Kaleidoscope of order $v$. More generally still, one might consider a set $\cal D$ of 2-$(n,h,1)$-designs with vertices all belonging to a given $v$-set $\cal V$, color the $b$ blocks of each design with $b$ different colors and ask once more that for any choice of two points of $\cal V$ and a color $c$ there is  only one design in which the block through the two chosen points has color $c$.

In this work, after the first definitions, properties and methods in Section \ref{basics}, we present in Section \ref{recur} some composition  constructions; 
Section \ref{ppower} contains our main theorem, an existence result stating that there is a {\it regular} Fano Kaleidoscope of order any admissible prime or prime power $v$, except $v=13$. 
This result  uses a construction from \cite{BP} ensuring existence for $v$ greater that an explicit bound; existence for smaller values of $v$ has been checked by computer.

In Section \ref{PBD} we show that the result on Kaleidoscopes of prime power order can be used to prove existence for a much larger class of orders, by using a construction involving pairwise balanced designs, and known results on their existence. Indeed, if we could construct a 
Kaleidoscope on 13 points, we would completely solve the existence problem for FK$(v)$s for all admissible values of $v$, with only 22 possible exceptions. 

As noted above, the notion of Kaleidoscope can easily be modified to be applied to designs other than the Fano plane; 
in Section \ref{hesse} we consider the notion of Hesse Kaleidoscope, where the role of the Fano plane is taken by the affine plane $AG(2,3)$. 
We shall discuss how the methods used previously can be used to obtain results also in this case, proving for instance the analogous result to Theorem \ref{FK(q)} for Hesse Kaleidoscopes, and  highlight the new difficulties that arise. 
 
Finally, we discuss further generalizations and outline some possible research directions in the last section.

\section{Fano Kaleidoscopes: notation and difference methods}\label{basics}

Let us start with the basic definitions, fix the notation we shall use in the paper, and present the techniques we shall use. 

As is well known, a 2-$(v,k,\lambda)$ {\it design}, or simply $(v,k,\lambda)$-{\it design}, is a pair $({\cal V},{\cal B})$, where $\cal V$ is a $v$-set of points and 
$\cal B$ is a set of $k$-subsets of $\cal V$, having the property that each pair of points of $\cal V$ is contained in precisely $\lambda$ blocks of  $\cal B$. 
The unique 2-$(7,3,1)$-design is called a {\it Fano plane}.
Let us repeat the definition of Fano Kaleidoscope and introduce some related notation.

\begin{definition}\label{defFK} Let ${\cal F}$ be a set of Fano planes such that
 the seven lines of each plane are colored with seven different colors
 $c_0,c_1,
 \dots, c_6$, and such that the points of the planes of ${\cal F}$ belong to a
 given $v$-set $\cal V$.  We say that $\cal F$ is a {\it Fano-Kaleidoscope} of order $v$ (briefly FK$(v)$) if for any
two distinct points $x$, $y$ of $\cal V$ and any color $c_i\in\{c_0,c_1,\dots,c_6\}$ there is exactly one Fano plane of ${\cal F}$ 
whose $c_i$-colored line contains $x$ and $y$.  
\end{definition}

Throughout this article, given an ordered 7-set $B=(b_0,b_1,\dots,b_6)$,  the triple $\ell_i(B)=\{b_i,b_{i+1},b_{i+3}\}$ (with the indices taken modulo 7) 
will be called the $i$-th line of $B$ for $0\leq i\leq 6$. Also, we set ${\cal L}(B)=\{\ell_i(B) \ | \ 0\leq i\leq 6\}$.
Of course the pair $(B,{\cal L}(B))$ is a Fano plane. 

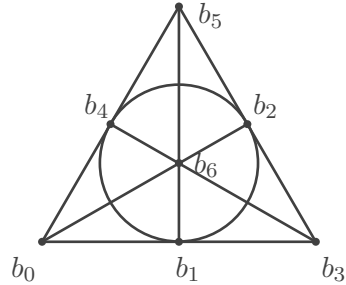
\begin{figure}[h]
\begin{tikzpicture}[line cap=round,line join=round,
x=1.2cm,y=1.2cm]
\clip(-2.0,1.30) rectangle (10,4.7);
\draw [color=uququq, line width=1] (2,2)-- (5,2);
\draw [color=uququq, line width=1] (5,2)-- (3.5,4.6);
\draw [color=uququq, line width=1] (3.5,4.6)-- (2,2);
\draw [color=uququq, line width=1] (3.5,4.6)-- (3.5,2);
\draw [color=uququq, line width=1] (2,2)-- (4.25,3.3);
\draw [color=uququq, line width=1] (2.75,3.3)-- (5,2);
\draw [color=uququq, line width=1] (3.5,2.87) circle (1.04cm);
\fill [color=uququq] (3.5,4.6) circle (1.5pt);
\draw[color=uququq] (1.8,1.7) node {$b_0$};
\draw[color=uququq] (3.6,1.7) node {$b_1$};
\draw[color=uququq] (5.2,1.7) node {$b_3$};
\draw[color=uququq] (4.45,3.5) node {$b_2$};
\draw[color=uququq] (2.6,3.5) node {$b_4$};
\draw[color=uququq] (3.85,4.5) node {$b_5$};
\draw[color=uququq] (3.8,2.85) node {$b_6$};

\fill [color=uququq] (2,2) circle (1.5pt);
\fill [color=uququq] (5,2) circle (1.5pt);
\fill [color=uququq] (3.5,2) circle (1.5pt);
\fill [color=uququq] (2.75,3.3) circle (1.5pt);
\fill [color=uququq] (4.25,3.3) circle (1.5pt);
\fill [color=uququq] (3.5,2.87) circle (1.5pt);
\end{tikzpicture}
\end{figure}

\begin{remark}\label{repeat}
First note that the underlying, uncolored structure of a FK$(v)$ is a 2-$(v,7,7)$ design. Thus, an obvious necessary condition for the existence of a FK$(v)$ is $v\equiv 1 \pmod{6}$.

Note also that a $FK(v)$ trivially exists for all $v$ for which a 2-$(v,7,1)$ design ${\cal D}$ exists: we fix an ordering of each  block $B$ of ${\cal D}$ and replicate each block  seven times, obtaining from each $B$ seven identical ordered blocks $B_0, B_1,\dots,B_6$. 
Then color the $i$-th line of $B$ with color $c_{i+j \ ({\rm mod} \ 7)}$ in $B_j$ for $0\leq i\leq 6$ and $0\leq j\leq 6$. 
The obtained structure clearly is the desired FK$(v)$.
\end{remark}

For instance, starting from the trivial 2-$(7,7,1)$ design whose only block is $B=\{0,1,2,3,4,5,6\}$ one gets the FK$(7)$ whose Fano planes $B_0$, $B_1$, \dots, $B_6$ all coincide with 
the pair $(B,{\cal L})$ where $${\cal L}=\bigl{\{}\{0,1,3\},\{1,2,4\},\{2,3,5\},\{3,4,6\},\{4,5,0\},\{5,6,1\},\{6,0,2\}\bigl{\}}$$ 
but whose lines are colored as displayed below 
\begin{center}
\begin{tabular}{|c|c|c|c|c|c|c|c|}
\hline
$$  & $B_0$ & $B_1$ & $B_2$ & $B_3$ & $B_4$ & $B_5$ & $B_6$\\
\hline
$\{0,1,3\}$ & $c_0$  & $c_1$  & $c_2$  & $c_3$  & $c_4$  & $c_5$  & $c_6$ \\
\hline
$\{1,2,4\}$ & $c_1$  & $c_2$  & $c_3$  & $c_4$  & $c_5$  & $c_6$  & $c_0$ \\
\hline
$\{2,3,5\}$ & $c_2$  & $c_3$  & $c_4$  & $c_5$  & $c_6$  & $c_0$  & $c_1$ \\
\hline
$\{3,4,6\}$ & $c_3$  & $c_4$  & $c_5$  & $c_6$  & $c_0$  & $c_1$  & $c_2$ \\
\hline
$\{4,5,0\}$ & $c_4$  & $c_5$  & $c_6$  & $c_0$  & $c_1$  & $c_2$  & $c_3$ \\
\hline
$\{5,6,1\}$ & $c_5$  & $c_6$  & $c_0$  & $c_1$  & $c_2$  & $c_3$  & $c_4$ \\
\hline
$\{6,0,2\}$ & $c_6$  & $c_0$  & $c_1$  & $c_2$  & $c_3$  & $c_4$  & $c_5$ \\
 \hline
\end{tabular}
\end{center}

Now we focus our attention on constructing {\it regular} Fano Kaleidoscopes via difference methods.

\begin{definition} Let $G$ be an additively written group of order $v$.  A Fano Kaleidoscope FK$(v)$ is called $G$-regular if $G$ acts sharply transitively on the set $\cal V$ of points (so that we can identify the elements of $\cal V$ with the elements of $G$), and if $G$ leaves it 
invariant.
\end{definition}

To build regular Fano Kaleidoscopes we may introduce a variation of the notion of {\em difference family}  which is used in constructing ``ordinary'' regular designs (see for instance \cite{AB,BJL}).

If $B$ is a subset of an additive group $G$, the list of
differences of $B$ is the multiset $\Delta B$ of all possible differences $x- y$ with $(x, y)$ an ordered pair of distinct elements of $B$. 
The development of $B$ under $G$ is the collection dev$B = \{B+g \ | \ g \in G\}$. 
Recall that a collection ${\cal F}$ of $k$-subsets of an additive group $G$ of order $v$ is a $(v, k, \lambda)$-difference family (DF for short), 
if the list of differences of ${\cal F}$, namely the multiset union
$\Delta {\cal F} := \displaystyle\biguplus_{B\in{\cal F}} \Delta B,$ covers every non-zero element of $G$ exactly 
$\lambda$ times.
 When it is necessary to specify the group $G$, we speak of a $(G, k, \lambda)$-DF. Such a difference family generates a regular 2-$(v, k, \lambda)$ design whose points
are the elements of $G$ and whose block-multiset is the development of ${\cal F}$,
namely the multiset sum dev${\cal F}:= \displaystyle\biguplus_{B\in{\cal F}} {\rm dev} B$.
We now adapt this definition to the kaleidoscopic situation.

\begin{definition}\label{defFKDF}
A \emph{Fano Kaleidoscopic Difference Family} of order $v=6t+1$, briefly FKDF$(v)$, is a $(v,7,7)$-difference family 
${\cal F}=\{B_1,\dots,B_t\}$ in a group $G$, together with an ordering of the seven points 
in each block $B_i$ of the family, such that for each fixed $j\in\{0,1,\dots,6\}$ the set ${\cal F}_j$ of the $j$-th lines 
of all blocks of $\cal F$
$${\cal F}_j=\{\ell_j(B_i) \ | \ 1\le i\le t\}$$ 
is a $(v,3,1)$-DF. When it is necessary to specify the group $G$, we will also speak of a FKDF$(G)$.
\end{definition}

It is easy to see from the definition of Fano Kaleidoscope that the following holds.

\begin{proposition}\label{FKDF}
The existence of  a FKDF$(v)$ in a group $G$ implies the existence of a  $G$-regular  FK$(v)$.
\end{proposition}
\begin{proof}
Let ${\cal F}=\{B_1,\dots,B_t\}$ be a FKDF$(v)$ in $G$. We already know that $dev{\cal F}$ is the collection of blocks of a 
2-$(v,7,7)$ design. We obtain the desired $G$-regular  FK$(v)$ by coloring the $j$-th triple of any translate of $B_i$ with
color $c_j$ for $1\leq i\leq t$ and $0\leq j\leq 6$.
\end{proof}

\begin{example}\label{19}
Consider the three 7-ples of elements of $\Z_{19}$  
$$B_1=(0,1,2,4,5,11,8), \ B_2=(0, 7, 14, 9, 16, 1, 18), \ B_3=(0, 11, 3, 6, 17, 7, 12)$$
and set ${\cal F}=\{B_1,B_2,B_3\}$.

Keeping the same notation as in Definition \ref{defFKDF} we have:
\small
$${\cal F}_0=\bigl{\{}\{0,1,4\},\{0,7,9\},\{0,11,6\}\bigl{\}};\quad {\cal F}_1=\bigl{\{}\{1,2,5\},\{7,14,16\},\{11,3,17\}\bigl{\}};$$
$${\cal F}_2=\bigl{\{}\{2,4,11\},\{14,9,1\},\{3,6,7\}\bigl{\}};\bigl{\}};\quad {\cal F}_3=\bigl{\{}\{4,5,8\},\{9,16,18\},\{6,17,12\}\bigl{\}};$$
$${\cal F}_4=\bigl{\{}\{5,11,0\},\{16,1,0\},\{17,7,0\}\bigl{\}};\quad {\cal F}_5=\bigl{\{}\{11,8,1\},\{1,18,7\},\{7,12,11\}\bigl{\}};$$
$${\cal F}_6=\bigl{\{}\{8,0,2\},\{18,0,14\},\{12,0,3\}\bigl{\}}.$$

\normalsize
Looking at the {\it difference tables} of the three blocks of ${\cal F}_0$

\begin{center}
\begin{tabular}{|l|c|r|c|r|c|r|c|r|c|r|c|r|}
\hline {$$} & 0 & 1 & 4    \\
\hline $0$ & $\bullet$ & \bf18 & \bf15  \\
\hline $1$ & $\bf1$ & $\bullet$ & \bf16  \\
\hline $4$ & $\bf4$ & \bf3 & $\bullet$  \\
\hline
\end{tabular}\quad\quad\quad
\begin{tabular}{|l|c|r|c|r|c|r|c|r|c|r|c|r|}
\hline {$$} & 0 & 7 & 9    \\
\hline $0$ & $\bullet$ & \bf12 & \bf10  \\
\hline $7$ & $\bf7$ & $\bullet$ & \bf17  \\
\hline $9$ & $\bf9$ & \bf2 & $\bullet$  \\
\hline
\end{tabular}\quad\quad\quad
\begin{tabular}{|l|c|r|c|r|c|r|c|r|c|r|c|r|}
\hline {$$} & 0 & 11 & 6    \\
\hline $0$ & $\bullet$ & \bf8 & \bf13  \\
\hline $11$ & $\bf11$ & $\bullet$ & \bf5  \\
\hline $6$ & $\bf6$ & \bf14 & $\bullet$  \\
\hline
\end{tabular} 
\end{center}
\normalsize
we see that $\Delta{\cal F}_0$ covers $\Z_{19}\setminus\{0\}$ exactly once so that ${\cal F}_0$ is a $(19,3,1)$-DF.
In the same way, one can see that the same happens for all the other ${\cal F}_j$s. Then we conclude that ${\cal F}$
 is a FKDF$(19)$.
\end{example}

\section{Composition constructions}\label{recur}
In this section, we present two composition constructions. The former is based on the notion of a {\it difference matrix} and the latter on the notion of a 
{\it pairwise balanced design}.

Recall that for $H$ an additive group, a $(H, k, 1)$ {difference matrix} (briefly DM) is a $k\times |H|$ matrix with entries from $H$ such that the difference of any two distinct rows contains each element of $H$ exactly once. 
We will make use of  the following well-known construction due to Jungnickel \cite{J} (see also \cite{Bdm}).

\begin{theorem}\label{dieter}
If we have a $(G,k,\lambda)$ difference family, a $(H,k,\lambda)$ difference family, and a $(H,k,1)$ difference matrix, then there exists a 
$(G\times H,k,\lambda)$ difference family.
\end{theorem}
\begin{proof}
Let ${\cal F}=\{B_i \ | \ i \in I\}$ be a $(G,k,\lambda)$-DF with $B_i=\{b_{i,1},\dots,b_{i,k}\}$ for each $i\in I$, let
${\cal F}'$ be a $(H,k,\lambda)$-DF, and let $M=[m_{r,c}]$ be a $(H,k,1)$-DM. 
Set $B_{i,j}=\{(b_{i,1},m_{1,j}),\dots,(b_{i,k},m_{k,j})\}$ for each pair $(i,j)\in I\times\{1,\dots,|H|\}$. Then
$${\cal F}\circ_M{\cal F}':=\{B_{i,j} \ | \ i\in I; 1\leq j\leq |H|\} \ \cup \ \{\{0\}\times B' \ | \ B' \in {\cal F}'\}$$
is the desired $(G\times H,k,\lambda)$-DF.
\end{proof}

This construction allows us to prove the following.

\begin{theorem}\label{composingFKDF}
If we have a  FKDF$(G)$, a FKDF$(H)$, and a $(H,7,1)$-DM, then there exists a FKDF$(G\times H)$.
\end{theorem}

\begin{proof}
Let us denote the three given ingredients by ${\cal F}$, $\cal F'$, and $M$, respectively.
We recall that ${\cal F}$ (resp. ${\cal F}'$) can be viewed as a $(G,7,7)$-DF whose blocks are ordered in such a way
that, for each $j$, the set ${\cal F}_j$ (resp. ${\cal F}'_j$) of all $j$-th lines
of ${\cal F}$ (resp. ${\cal F}'$) is a $(G,3,1)$ (resp. $(H,3,1)$) difference family.

For $1\leq j\leq 7$, denote by $M_j$ the $3\times |H|$ matrix obtained from $M$ by selecting the rows $j$, $j+1$, and $j+3$ (mod 7).
Note that $M_j$ is a $(H,3,1)$-DM since, 
in general, any set of $k'$ distinct rows of a $(H,k,1)$-DM form a $(H,k',1)$-DM. 
Consider the $(G\times H,7,7)$ difference family ${\cal F}\circ_M{\cal F}'$
constructed as in the proof of Proposition \ref{dieter}.
By the same proposition, one can see that for $1\leq j\leq 7$, the set of all the $j$-th lines
of ${\cal F}\circ_M{\cal F}'$ is the $(G\times H,3,1)$ difference family ${\cal F}_j\circ_{M_j}{\cal F}'_j$.
This proves that ${\cal F}\circ_M{\cal F}'$ is a FKDF$(G\times H)$. 
\end{proof}

In the next section the above theorem will be applied several times in conjunction with the
following well-known remark.

\begin{remark}\label{(q,7,1)-DM}
There exists a $(\F_q,k,1)$-DM for all prime powers $q\geq k$. Hence, in particular, we
have a $(\F_q,7,1)$-DM for any prime power $q\geq7$.
\end{remark}

We now present a very useful construction relying on pairwise balanced designs
which might be thought of as a generalization of Remark \ref{repeat}.
Recall that a $(v,K,1)$-{\it pairwise balanced design}, or $(v,K)$-PBD, is a pair $({\cal V},{\cal B})$ 
where $\cal V$ is a $v$-set of points and $\cal B$ is a set of blocks, with $|B|\in K$ for all $B\in \cal B$, 
having the property that each pair of points of $\cal V$ is contained in  exactly one block of $\cal B$; when $K=\{k\}$, this gives a 2-$(v,k,1)$ design. 

\begin{proposition} 
If there exists a $(v,K,1)$-PBD and a FK$(k)$ for all $k\in K$, then there exists a FK$(v)$. 
\end{proposition}
\begin{proof}
By assumption, one can build a Fano Kaleidoscope on the points of each block of the PBD. It is clear that putting all these FKs together,  one obtains
the required FK$(v)$.
\end{proof}

\section{The case $v$ prime or a prime power}\label{ppower}
In the following, given a prime power $q$, the additive and multiplicative groups of the field of order $q$ will be denoted by $\F_q$ and $\F_q^*$,
respectively. If $q\equiv1$ (mod 6), we denote by $C^3$ the group of non-zero cubes of $\F_q$, namely the subgroup of $\F_q^*$ of index $3$.
The cosets of $C^3$ in $\F_q^*$ - also called the {\em cyclotomic classes of index $3$} - will be denoted by $C_0^3$ (that is $C^3$ itself), $C_1^3$, and $C_2^3$. 

We need the following result which is a special consequence of the well-known {\it lemma 
on evenly distributed differences} by R.M. Wilson \cite{W} (see also \cite{BJL}, Lemma 6.3, p. 500).

\begin{lemma}\label{wilson}
Assume that $q=6n+1$ is a prime power and that $\ell$ is a triple of elements of $\F_q$ such that $\Delta \ell=\{1,-1\}\cdot T$ 
with $T$ having exactly one element in each cyclotomic class of index $3$. Then, denoted by $S$ an arbitrary set of representatives for the cosets of $\{1,-1\}$
in $C^3$, we have that $\bigl{\{}\{sa,sb,sc\} \ | \ s\in S\bigl{\}}$ is a $(\F_q,3,1)$-DF.
\end{lemma}
We also need an asymptotic result of the first author and A. Pasotti (Theorem 2.2.  in \cite{BP}) which 
modified and specialized to our situation has the following statement.
\begin{lemma}\label{BP} 
Let $q\equiv 1 \pmod{6}$ be a prime power and let $t$ be a positive integer.
Then, for any $t$-subset $C=\{c_1,\dots,c_t\}$ of $\F_q$ and for any ordered $t$-tuple $(\gamma_1,\dots,\gamma_t)$ of $\Z_3^t$,
the set $X:=\{x\in \F_q: x-c_i\in C_{\gamma_i}^3 \,\,{\rm  for }\,\, i=1,\dots,t \}$
is not empty as soon as $q$ is greater than a bound $Q(t)$ which is an increasing function of $t$. 
\end{lemma}

The bound $Q(t)$ increases dramatically with $t$. In the table below we 
report its exact value - easily deducible from \cite{BP} - for $1\leq t\leq 8$:
\small
\begin{center}
\begin{tabular}{|c|c|c|c|c|c|c|c|c|c|c|cl}
\hline
$\bf t$ & $1$ & $2$ & $3$ & $4$& $5$ & $6$& $7$ & $8$\\
\hline
$\bf Q(t)$ & $1$ &36 & 939 & 19,350 & 326,661 & 4,790,260 & 64,391,800 & 808,659,000 \\ 
 \hline
\end{tabular}
\end{center}

\normalsize

We are now ready to prove the crucial result of this section.

\begin{theorem}\label{FK(q)} A regular {\rm FK}$(q)$ exists whenever $q$ is a prime or a prime power 
congruent to $1$ $($mod $6)$ provided that $q\ne 13$.\end{theorem}

\begin{proof}
Let $q\equiv1$ (mod 6) be a prime power.
To build a FK$(q)$, it is enough to show that there exists a special block ({\em initial block}), 
$B=(b_0,b_1,\dots,b_6)$
with the property that the $i$-th line of $B$ satisfies the condition of Lemma \ref{wilson}
for $0\leq i\leq 6$.
Indeed in this case it is clear that 
$$\{(sb_0,sb_1,sb_2,sb_3,sb_4,sb_5,sb_6) \ | \ s\in S\}$$
is a FKDF$(q)$ for any arbitrarily chosen set $S$ of representatives for the cosets of $\{1,-1\}$
in $C^3$, and then the assertion follows from Proposition \ref{FKDF}.

We prove the existence of our initial block $B$ distinguishing the two cases $q>Q(5)$ that we
solve with the use of Lemma \ref{BP}, and $q<Q(5)$ that we essentially solved by computer search. 

\medskip\noindent
\underline{Case 1}: $q>Q(5)$.

This case heavily relies on Lemma \ref{BP}. It is convenient to split it into two subcases 
according to whether 2 is or is not a cube in $\F_q$.

\medskip\noindent
\underline{Subcase 1a}: $q>Q(5)$ and $2\in C^3$.
 
The set 
$$X=\bigl{\{}x\in \F_q \ : \ \{x, x+1\}\subset C^3_1; \ x-1\in C^3_2\bigl{\}}$$
is not empty for $q>Q(3)$ by Lemma \ref{BP}. Hence, a fortiori, we have $X\neq\emptyset$ in view of our assumption that $q>Q(5)$.
Fix any element $\overline x\in X$ and consider the set
$$Y=\bigl{\{}y\in \F_q \ : \ \{y+1, \  y+\overline x\}\subset C^3_0; \ y-1\in C^3_1; \ \{ y, \ y-\overline x\}\subset C^3_2\bigl{\}}.$$ 
This set is not empty for $q>Q(5)$ by Lemma \ref{BP} again. Fix any element $\overline y\in Y$ and consider 
the 7-tuple $$B=(0, \ 1, \ -1, \ \overline x, \ -\overline x, \ \overline y, \ -\overline y).$$
The seven lines of $B$ are:
$$\ell_0=\{0,1,\overline x\};\quad \ell_1=\{1,-1,-\overline x\};\quad \ell_2=\{-1,\overline x, \overline y\};\quad \ell_3=\{\overline x,-\overline x,-\overline y\};$$
$$\ell_4=\{-\overline x,\overline y,0\};\quad \ell_5=\{\overline y, -\overline y,1\};\quad \ell_6=\{-\overline y,0,-1\}.$$
We have $\Delta \ell_i=\{1,-1\}\cdot T_i$ with:
$$T_0=\{1,\overline x,\overline x-1\};\quad T_1=\{2,\overline x+1,\overline x-1\};\quad T_2=\{\overline x+1,\overline y+1,\overline y-\overline x\};$$
$$T_3=\{2\overline x,\overline y+\overline x,\overline y-\overline x\};\quad
T_4=\{\overline y+\overline x,\overline x,\overline y\};\quad
T_5=\{2\overline y,\overline y-1,\overline y+1\};$$
$$T_6=\{1,\overline y,\overline y-1\}$$
The reader can easily see that our hypotheses that $2\in C^3$, $\overline x\in X$ and $\overline y\in Y$ imply that  
$\ell_i$ satisfies the assumption of Lemma \ref{wilson} and then $B$ is the desired initial block of our FK$(q)$.

\medskip\noindent
\underline{Subcase 1b}: $q>Q(5)$ and $2\notin C^3$.
 
We have $2\in C^3_i$ with $i=1$ or 2.
Take, as you like, an element $\overline x$ belonging to the set 
$$X=\bigl{\{}x\in \F_q \ : \ x+1\in C^3_0; \quad x\in C^3_i; \ x-1\in C^3_{2i}\bigl{\}}$$
which is not empty by Lemma \ref{BP}.
Then take any element $\overline y$ of the set
$$Y=\bigl{\{}y\in \F_q \ : \  y+\overline x\in C^3_0; \ \{y-1, y-\overline x\} \subset C^3_i; \ \{y, \  y+1\}\subset C^3_{2i}\bigl{\}}$$
which is also not empty by Lemma \ref{BP}.
As in the previous case the reader can easily check that an initial block of our FK$(q)$ is given by
the 7-tuple\break $B=(0, \ 1, \ -1, \ \overline x, \ -\overline x, \ \overline y, \ -\overline y)$.

\medskip\noindent
\underline{Case 2}: $q<Q(5)$.

Here we split our search for a FKDF$(q)$ into five subcases according to whether: 
\begin{itemize}
\item [a)] $q$ is a prime;
\item [b)] $q=p^n$ with  $13\neq p\equiv1$ (mod 6) a prime; 
\item [c)] $q=p^{2n}$ with $p\equiv5$ (mod 12) a prime; 
\item [d)] $q=p^{2n}$ with $p\equiv11$ (mod 12) a prime;
\item [e)] $q=13^n$ with $n\geq2$.
\end{itemize}

\medskip\noindent
\underline{Subcase 2a}: $q$ is a prime $p>13$.

To find a FKDF$(p)$ for all admissible primes $p$ in the range $]13,Q(5)]$,
it is computationally convenient to look for an initial block of the form 
$$B(x):=(0,1,2,x,x+1,x^2+x,2x)$$ 
for  a suitable choice of $x\in\Z_p$.  Indeed, for a block of this form it is enough to check the cyclotomic conditions 
on the differences of only three of the seven possible lines, namely $\ell_0=\{0,1,x\},\ell_2=\{2,x,x^2+x\},
\ell_5=\{x^2+x,2x,1\}$, since we have: $$\Delta \ell_0=\Delta \ell_1=\Delta \ell_3;\quad\quad \Delta \ell_4=(x+1)\cdot\Delta \ell_0;\quad\quad
\Delta \ell_6=2\cdot\Delta \ell_0.$$ 

An element $x$ having the property that  $B(x)$ is an initial block has been 
found by computer search for all admissible primes $p<Q(5)$, except for those belonging to the list 
$L=\{13,19,31,43,61,79,127,199\}$. 
Here is a table of the value of $x$ found for almost all primes $p$ up to 600.  
We omitted the primes $p\equiv1$ (mod 42) greater than 127 since for each of these primes a 2-$(p,7,1)$ design is known \cite{AG} and then
the existence of a FK$(p)$ is already guaranteed by Remark \ref{repeat}.

\begin{center}
\begin{tabular}{|c|c|||c|c|||c|c|||c|c|||c|c|}
\hline
$p$ & $x$ & $p$ & $x$& $p$ & $x$& $p$ &$x$ &$p$ &$x$\\
\hline
37&13 & 67&61 & 73&35 & 97&5 & 103&18\\ 
\hline
109&26 & 139&47 & 151&12 & 157&84 & 163&55\\
\hline
181&61 & 193&78 & 223&143 & 229&37 & 241&20\\
\hline
271&89 & 277&47 & 283&7 & 307&23 & 313&92\\
\hline
 331&48 & 349&55 & 367&34 & 373&122 & 397&19\\
 \hline
 409&37 & 433&24 & 439&174 & 457&147& 487&111\\
 \hline
 499&87 & 523&133 & 541&10 & 571&3 & 577&80\\
 \hline
\end{tabular}

\end{center}

An exhaustive search has ruled out the existence of a regular FK$(13)$.
An example of a FKDF$(19)$ has been already given in Example \ref{19}.
For the remaining primes in $L$ we found an alternative initial block $B$ as follows.

\begin{center}

\begin{tabular}{|c|c|}
\hline
$p$  & $B$\\
\hline
31 & $(0,1,2,12,13,27,24)$ \\
43& $(0,1,2,7,8,37,38)$\\
61& $(0,1,2,5,6,41,10)$\\
79& $(0,1,2,24,25,11,48)$\\ 
127& $(0,1,2,12,13,87,24)$\\ 
199& $(0,1,2,4,5,71,8)$\\ 
 \hline
\end{tabular}

\end{center}

\medskip\noindent
\underline{Subcase 2b}: $q=p^n$ with $p$ prime, $13\neq p\equiv1$ (mod 6).

Here we already know that there exists a FKDF$(p)$ from Subcase 2a. Then, by Remark \ref{(q,7,1)-DM} and by iterated application of
Theorem \ref{composingFKDF} we have a FKDF$(p^n)$ in $\Z_p^n$ for any $n$.

\medskip\noindent
\underline{Subcase 2c}: $q=p^{2n}$ with $p$ prime, $p\equiv5$ (mod 6).

It is enough to prove the existence of a FKDF$(p^2)$ in $\Z_p^2$ thought of as the additive group of $\F_{p^2}$.
Indeed, under this assumption,  then one gets a FKDF$(p^{2n})$
in $\Z_p^{2n}$ by Remark \ref{(q,7,1)-DM} and the iterated use of Theorem \ref{composingFKDF}.
Also, by the assumption $q<Q(5)$, we can limit our research to primes $p<\sqrt{Q(5)}$.
Thus we have to work in the finite field $\F_{p^2}$ with $p<571$. 
Note that by the law of quadratic reciprocity (see, e.g., \cite{L}), 3 is not a square (mod $p$), hence $\F_{p^2}$ can be identified  with $\Z_p[t]/(t^2-3)$.
Here it is convenient to look  for an initial block  of the form $B=(1,x,x^2,x^3,x^4,x^5,x^6)$ for a suitable $x\in \F_{p^2}$ since once more we need to
check the cyclotomic conditions of the differences of only three lines of $B$ that here are $\ell_0(B)=\{1,x,x^3\}$, $\ell_4(B)=\{x^4,x^5,1\}$, 
and $\ell_6(B)=\{x^6,1,x^2\}$. Such an $x$ 
has been always found by computer search in the range we are interested in, as shown in the table below.

\small
\begin{center}
\begin{tabular}{|C|C||||C|C||||C|C||||C|C|}
\hline
p & x & p & x& p & x& p & x\\
\hline
5&4+t        & 113 & 1+39t  & 269& 1+65t  & 449 & 1+8t\\ 
\hline
17&6+3t     & 137 & 1+63t      & 281& 1+7t  & 461 & 1+8t   \\
\hline
29&1+2t    &  149 & 1+17t   & 293& 3+9t  & 509 & 1+103t   \\
\hline
41&3+15t & 173 &1+34t & 317& 1+27t  & 521 & 1+82t\\
\hline
53&1+19t & 197& 2+18t  &  353& 1+9t & 557 & 1+7t \\
\hline
89&1+15t & 233& 1+99t & 389& 1+11t & 569 & 1+116t\\
\hline
101&1+43t & 257 & 1+33t & 401& 1+40t & &\\
\hline
\end{tabular}
\end{center}

\normalsize
\medskip\noindent
\underline{Subcase 2d}: $q=p^{2n}$ with $p$ prime, $p\equiv11$ (mod 6).

For the same reasons as in the previous subcase it is enough to find a FKDF$(p^2)$ in the additive group of $\F_{p^2}$
for $p\leq571$. 
Here we have $p\equiv3$ (mod 4), hence $-1$ is not a square (mod $p$) and then $\F_{p^2}$ can be identified  with $\Z_p[t]/(t^2+1)$.
Once more we looked for an initial block  of the form $(1,x,x^2,x^3,x^4,x^5,x^6)$ for a suitable $x\in \F_{p^2}$, and such an $x$ 
has been always found by computer search in the range we are interested in.
Here are our computer results.

\small
\begin{center}

\begin{tabular}{|C|C||||C|C||||C|C||||C|C|}
\hline
p & x & p & x& p & x& p &x \\
\hline
11&3+4t       & 131 & 1+22t  & 263&1+56t   & 443&1+122t   \\ 
\hline
23&1+11t     & 167 &3+9t  &  311&2+41t    & 467&1+31t\\
\hline
47&2+12t     &  179&1+8t   &  347&1+16t  & 479&1+103t\\
\hline
59 &2+15t   & 191&1+23t  &  359&1+157t & 491&1+126t\\
\hline
71&2+32t     & 227&1+91t  &383 & 1+122t& 503&1+50t\\
\hline
 83&2+3t      & 239&1+101t &419 & 1+30t& 563&1+73t\\
 \hline
 107&2+51t   &251 &1+42t  &431 & 1+15t& &\\
 \hline
\end{tabular}
\end{center}

\normalsize

\medskip\noindent
\underline{Subcase 2e}: $q=13^n$ with $n>1$.
 
Identifying $\F_{13^2}$ with $\Z_{13}[t]/(t^2-2)$, it is possible to check that the 7-ple $(0,1,2,x,x+1,x^2+x,2x)$, $x=6+2t$,
is the initial block for a FKDF$(13^2)$
in the additive group of $\F_{13^2}$.
Also, if we see $\F_{13^3}$ as $\Z_{13}[t]/(t^3-2)$, then the 7-ple 
$(1,x,x^2,x^3,x^4,x^5,x^6)$, $x=10+7t+11t^2$, is the initial block for a FKDF$(13^3)$
in the additive group of $\F_{13^3}$.

Now, for the existence of a FKDF$(13^n)$ for $n\geq4$, it is enough to combine iteratively the two FKDFs above using Remark \ref{(q,7,1)-DM}
and Theorem \ref{composingFKDF}.

This completes the proof.
\end{proof}

\begin{remark}
Note that, when $q\equiv 3\pmod{4}$, we might take the group $C^6$ of sixth powers in $\F_q^*$ as the  transversal set $S$ for the cosets of $\{1,-1\}$ in $C^3$ required in the proof; in this case the set of blocks is $${\cal B}=\{m{B}+\tau, m\in C^6, \tau\in \F_q\},$$ so that $\F_q \rtimes C^6$ is an automorphism group of our FK$(v)$ which acts sharply transitively on the blocks. We are in this situation for instance when $q=19$ (Example \ref{19}), where the group $C^6$ is $\{1,2^6,2^{12}\}=\{1,7,11\}$ and the initial block is ${B}=(0,1,2,4,5,11,8)$.
\end{remark}

Let us mention the curious fact that the 7-ple $A=(0,1,2,3,4,5,6)$  can be an initial block for a FKDF$(p)$, $p$ prime. 
The lines of $A$ are $\ell_1=\{0,1,3\}$, $\ell_2=\{1,2,4\}$, $\ell_3=\{2,3,5\}$, $\ell_4=\{3,4,6\}$,
$\ell_5=\{5,6,1\}$, $\ell_6=\{6,0,2\}$ and their lists of differences are the following:
$$\Delta\ell_1=\Delta\ell_2=\Delta\ell_3=\Delta\ell_4=\{1,-1\}\cdot\{1,2,3\};$$
$$\Delta\ell_5=\{1,-1\}\cdot\{1,4,5\};\quad \Delta\ell_6=\{2,-2\}\cdot\{1,2,3\}.$$
It follows that $A$ is the initial block of a FKDF$(p)$ if both the triples $\{1,2,3\}$ and $\{1,4,5\}$
are systems of representatives for the cyclotomic classes of index 3. The first triple, $\{1,2,3\}$,
is such a system provided that $2\in C^3_i$ and $3\in C^3_{2i}$ with $i=1$ or $2$.
Now, if $2\in C^3_i$ - which implies $4\in C^3_{2i}$ - the triple $\{1,4,5\}$ is a  
system of representatives for the cyclotomic classes of index 3 provided that $5\in C^{3}_i$.
We conclude that $A$ is the initial block of a FKDF$(p)$ when $\{2,5\}\subset C^3_i$ and $3\in C^3_{2i}$ for $i=1$ or 2.
These conditions can be equivalently restated as follows.

\begin{proposition}
Let 
$p\equiv1$ $($mod $6)$ 
be a prime such that $2$ is not a cube while both $6$ and $20$ are cubes modulo $p$.
Then $A=(0,1,2,3,4,5,6)$ is the initial block of a FKDF$(p)$. 
\end{proposition}

The set of primes up to $1000$ satisfying the conditions of the previous proposition is  $\{7, 541,$ $ 571, 877, 937\}$.

\section{Recursive constructions}\label{PBD}

In this section we denote by $Q_{1(6)}$ the set of all prime powers congruent to 1 modulo 6.

By Theorem \ref{FK(q)}, Remark \ref{(q,7,1)-DM}, and the iterated 
use of Theorem \ref{dieter},  we have the following result.

\begin{theorem}\label{recursive}
If all maximal prime-power factors of an integer $v$ belong to  $Q_{1(6)}\setminus\{13\}$,
then there exists a regular FK$(v)$.
\end{theorem} 

The following result \cite{MS,Gr} emphasizes the significance of Theorem \ref{FK(q)}.

\begin{theorem}\label{MS}
There exists a $(v, Q_{1(6)},1)$-PBD for all $v$ congruent to $1 \pmod{6}$ with the following $22$ exceptions

\medskip\noindent\footnotesize
$55, 115, 145, 205, 235, 265, 319, 355, 391, 415, 445, 451, 493, 649, 667, 685, 697,745, 781, 799, 805, 1315;$

\medskip\normalsize\noindent
of these, $55$ is a definite exception and the other values are possible exceptions.
\end{theorem}

\begin{remark}
In light of Theorem \ref{FK(q)}, we have existence of a regular FK$(q)$ for all $q\in Q_{1(6)}\setminus\{13\}$;
the existence of a (necessarily not regular) FK$(13)$ is still open, but if we could construct such a Kaleidoscope, then we would immediately have  the existence of Fano Kaleidoscopes for all admissible values of $v$ with the possible exception of the 22 values listed in Theorem \ref{MS}.
Note that the underlying structure of a  $FK(13)$ is a  2-$(13,7,7)$ design, the block complement of a 2-$(13,6,5)$ design;  these 	 have been enumerated by Kaski and \"Osterg{\aa}rd in \cite{KO} - there are  $19,072,802$ designs with these parameters. 
\end{remark}


In any case, there are results of \cite{MS} giving a $(v,K,1)$-PBD with $v\equiv1$ (mod 6)
which do not require that $13\in K$. These results can already be applied to constructing FKs. 
For instance from Lemma 3.4 and 3.5 of \cite{MS} we have the existence of a 
FK$(v)$ for $v\in \{187, 385,1537\}$; from Corollary 2.6 of \cite{MS} we obtain the following more general result.

\begin{corollary} 
%
Let $m$ be an integer whose maximal prime-power factors are not smaller than $43$.
If there exists a FK$(m)$ and a FK$(m+6t)$ with $0\leq t\leq m$, then there exists a FK$(43m+6t)$.
\end{corollary} 
For instance, when taking $m=43$ we obtain the existence of a FK$(v)$ for the following new values of $v$:
\small
$$1885, 1903, 1909, 1915, 1927, 1945, 1963, 1969, 1975, 2005, 2035, 2047, 2065, 2095.$$
\normalsize
These are all $v$'s of the form $43\cdot43+6t$ not belonging to $Q_{1(6)}$ with $t\leq43$ and $43+6t\in Q_{1(6)}$.

\section{Hesse Kaleidoscopes}\label{hesse}

It is well-known that, up to isomorphism, there is only one 2-$(9,3,1)$-design. This is the affine plane on the field of order 3 usually denoted by AG$(2,3)$.
In the following any such design will be called a {\it Hesse plane}\footnote{This terminology is suggested 
by the fact that the configuration of points and lines of AG$(2,3)$ is sometimes called the {\it Hesse configuration} (see, e.g., \cite{cox}).}.
It is natural to generalize the notion of a Fano Kaleidoscope to that of a {\it Hesse Kaleidoscope} as follows.

\begin{definition}\label{defHK} Let ${\cal H}$ be a set of Hesse planes such that
 the twelve lines of each plane are colored with twelve different colors
 $c_0,c_1, \dots, c_{11}$, and such that the points of the planes of ${\cal H}$ belong to a
 given $v$-set $\cal V$.  We say that $\cal H$ is a {\it Hesse-Kaleidoscope} of order $v$ (briefly HK$(v)$) if for any
two distinct points $x$, $y$ of $\cal V$ and any color $c_i\in\{c_0,c_1,\dots,c_{11}\}$ there is exactly one Hesse plane of ${\cal H}$ 
whose $c_i$-colored line contains $x$ and $y$.  
\end{definition}

The underlying structure of a HK$(v)$ is clearly a 2-$(v,9,12)$ design, and therefore  the admissibility conditions  
give $v\equiv  1$ or 3 (mod 6) as a necessary condition for the existence of a HK$(v)$.
Using the {\it $1$-rotational representation} of AG$(2,3)$, it is convenient to represent each Hesse plane $B$ of a
HK$(v)$ with point-set $\cal V$ as an ordered 9-ple of distinct elements of $\cal V$
$${B}=(b_\infty,b_0,b_1,b_2,b_3,b_4,b_5,b_6,b_7)$$
labelled so that the lines of $B$ are given by the 12-ple $${\cal L}(B)=(\ell_0(B),\ell_1(B),\dots,\ell_{11}(B))$$
where
$$\ell_i(B)=\begin{cases}\{b_i,b_{i+1},b_{i+3}\} \quad\,\mbox{ for $0\leq i\leq 7$;}\cr
\{b_\infty,b_i,b_{i+4}\} \quad\quad\mbox{for $8\leq i\leq 11$}\end{cases}$$
and where the indices distinct from $\infty$ have to be understood modulo 8.

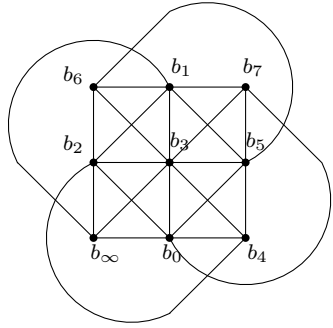
\begin{figure}[htb]
\begin{tikzpicture}[line cap=round,line join=round,>=triangle
45,x=1.0cm,y=1.0cm]
\clip(-2.7,0.14) rectangle (10.06,5.2);
\draw (2,4)-- (3,4);
\draw (3,4)-- (4,4);
\draw (2,3)-- (3,3);
\draw (3,3)-- (4,3);
\draw (2,2)-- (4,2);
\draw (2,4)-- (2,2);
\draw (3,4)-- (3,2);
\draw (4,4)-- (4,2);
\draw (2,2)-- (4,4);
\draw (2,4)-- (4,2);
\draw (4,3)-- (3,4);
\draw (2,3)-- (3,2);
\draw (3,2)-- (4,3);
\draw (2,3)-- (3,4);
\draw (4,2)-- (3,1);
\draw [shift={(2.5,2)}] plot[domain=2.03:5.18,variable=\t]({1*1.12*cos(\t
r)+0*1.12*sin(\t r)},{0*1.12*cos(\t r)+1*1.12*sin(\t r)});
\draw (4,4)-- (5,3);
\draw [shift={(4,2.5)}] plot[domain=-2.68:0.46,variable=\t]({1*1.12*cos(\t
r)+0*1.12*sin(\t r)},{0*1.12*cos(\t r)+1*1.12*sin(\t r)});
\draw (2,4)-- (3,5);
\draw [shift={(3.5,4)}] plot[domain=-1.11:2.03,variable=\t]({1*1.12*cos(\t
r)+0*1.12*sin(\t r)},{0*1.12*cos(\t r)+1*1.12*sin(\t r)});
\draw (2,2)-- (1,3);
\draw [shift={(2,3.5)}] plot[domain=0.46:3.61,variable=\t]({1*1.12*cos(\t
r)+0*1.12*sin(\t r)},{0*1.12*cos(\t r)+1*1.12*sin(\t r)});
\begin{scriptsize}
\fill [color=black] (2,2) circle (1.5pt);
\draw[color=black] (2.16,1.80) node {$b_\infty$};
\fill [color=black] (3,2) circle (1.5pt);
\draw[color=black] (3.04,1.80) node {$b_0$};
\fill [color=black] (4,2) circle (1.5pt);
\draw[color=black] (4.16,1.80) node {$b_4$};
\fill [color=black] (2,3) circle (1.5pt);
\draw[color=black] (1.74,3.24) node {$b_2$};
\fill [color=black] (3,3) circle (1.5pt);
\draw[color=black] (3.14,3.28) node {$b_3$};
\fill [color=black] (4,3) circle (1.5pt);
\draw[color=black] (4.14,3.28) node {$b_5$};
\fill [color=black] (2,4) circle (1.5pt);
\draw[color=black] (1.74,4.16) node {$b_6$};
\fill [color=black] (3,4) circle (1.5pt);
\draw[color=black] (3.16,4.26) node {$b_1$};
\fill [color=black] (4,4) circle (1.5pt);
\draw[color=black] (4.1,4.26) node {$b_7$};
\end{scriptsize}
\end{tikzpicture} 
\end{figure}

Reasoning as in Remark \ref{repeat} we can say that a $HK(v)$ exists whenever
a 2-$(v,9,1)$ design exists, namely for all $v\equiv 1$ or 9 (mod 72) with the 87 exceptions in Tables 3.11, 3.12 in \cite{AG}. 

A {\it regular} HK$(v)$ with $v\equiv1$ (mod 6) can be built from a \emph{Hesse Kaleidoscopic Difference Family} of order $v$, briefly HKDF$(v)$.
This is a $(v,9,12)$-DF in some group whose blocks are ordered 
in such a way that the set of the $i$-th lines of all its blocks is a $(v,3,1)$-DF for $0\leq i\leq 11$. 

As an example, consider the 9-tuple $B=(0, 1, 2, 3, 7, 16, 8, 4, 10)$ of elements of $\Z_{19}$. Then one can check that
${\cal F}:=\{B,7B,11B\}$ is a HKDF$(19)$.

Note that we can use composition techniques similar to the ones used in the Fano case, since the following result is easily proved.
\begin{theorem}\label{composingHKDF}
If we have a  HKDF$(G)$, a HKDF$(H)$, and a $(H,9,1)$-DM, then there exists a FKDF$(G\times H)$.
\end{theorem}

 The following theorem is the analogous result to Theorem \ref{FK(q)}, and it is proved using the same ideas and methods. Let us point out though that we now will need to settle by computer the existence of a HK$(v)$ for the admissible prime power values of $v$ in the far bigger interval $[7,Q(8)]$, with $Q(8)=808,659,000$.

\begin{theorem}\label{HK(q)} A regular  {\rm HK}$(q)$ exists for any prime or prime power
$q\equiv1$ $($mod $6)$,  $q \ne13$.
\end{theorem}

\begin{proof}
We will sketch the proof, which can be obtained along the lines of Theorem \ref{FK(q)}. It is enough to prove that
for  each prime power $q$ as in the statement there is a HKDF$(q)$ generated by an initial block
$B=(b_\infty,b_0,b_1,\dots,b_7)$ with the property that each line of $B$ satisfies the condition of Lemma \ref{wilson}.

Its existence can be proved using Lemma \ref{BP} if $q$ is greater than $Q(8)$; a computer search dealt with the cases $q\in [19, Q(8)]$.
When $q>Q(8)$ we can build the required  $B=(b_\infty,b_0,b_1,\dots,b_7)$ taking $(b_\infty,b_0,b_1,b_2)=(0,1,2,3)$
and each of the subsequent elements $b_3,\dots, b_7$ obtained, consecutively, by choosing 
$b_k$ in an appropriate set $X_k$ as indicated below, where $C^3_i$ and $C^3_j$ denote the cyclotomic 
classes containing 2 and 3, respectively.

\scriptsize
$$X_3=\{x\in \F_q \ : \{x,x-3\}\subset C^3_0; x-1\in C^3_1; x-2\in C^3_2\}$$
$$X_4=\{x\in \F_q : x-b_3\in C^3_0; \{x,x-2\}\subset C^3_1;
\{x-1, x-3\}\subset C^3_2\}.$$
$$X_5=\{x\in \F_q : x\in C^3_{i+1}; x-2\in C^3_{i+2}; x-b_4\in C^3_0; \{x-1,x-3\}\subset C^3_1; x-b_3\in C^3_2\}.$$
$$X_6=\{x\in \F_q : x\in C^3_{j+1}; x-3\in C^3_{j+2}; x-b_5\in C^3_0; \{x-2,x-b_3\}\in C^3_1; \{x-1,x-b_4\}\in C^3_2\}.$$
$$X_7=\{x\in \F_q : x-b_6\in C^3_0; \{x,x-b_4\}\subset C^3_1; \{x-2,x-b_3,x-b_5\}\subset C^3_2;x-1\in C^3_{i+1}; x-3\in C^3_{i+2}\}.$$

\normalsize

As mentioned above, the cases $v\in [19, Q(8)]$ have been solved by computer search. 
As in the proof of Theorem \ref{FK(q)}, one first considers the case $v=p$ a prime. For prime orders,
we have found an initial block of the form $B(x)=(0,1,x,x^2,x^3,x^4,x^5,x^6,x^7)$ for a suitable $x\in \Z_p$ for almost all $p$ in the required interval.
This approach is convenient since in order to check whether such a block $B(x)$ is good, it is enough to
check the differences of only four of its lines, namely $\{1,x,x^3\}$, $\{x^5,x^6,1\}$, $\{x^7,1,x^2\}$ and $\{0,1,x^4\}$.
We list here the first cases where this approach has been successful.

\begin{center}
\begin{tabular}{|c|c|c|c|c|c|c|c|c|}
\hline
$p$  & 97 & 103 & 139 & 163 & 181 & 223 & 229 & 277\\
\hline
$x$ & 14 & 36 & 61 & 143 & 66 & 187 & 184 & 97 \\
 \hline
\end{tabular}
\end{center}

The only twenty-six primes $p\in [19, Q(8)]$ for which we did not find an initial block $B(x)$ as described above are the following:
13, 31, 37, 43, 61, 67, 73, 79, 109, 127, 151, 157, 193, 199, 211, 241, 271,
283, 337, 349, 367, 463, 733, 751, 811, 937.
We found an initial block for these exceptional primes anyway, apart from the case $p=13$.
The following table presents the initial block for $p<100$. 

\begin{center}
\begin{tabular}{|c|c|}
\hline
$p$  & $B$\\
\hline
31 & $(12, 0, 1, 3, 6, 13, 8, 28, 11)$ \\
37& $(24, 0, 1, 7, 3, 35, 29, 25, 17)$\\
43& $(13, 0, 1, 3, 7, 8, 22, 17, 14)$\\
61& $(50, 0, 1, 6, 5, 15, 10, 13, 14)$\\
67& $(26, 0, 1, 6, 7, 18, 13, 12, 11)$\\
73& $(3, 0, 1, 4, 6, 29, 27, 16, 17)$\\ 
79& $(16, 0, 1, 4, 20, 12, 25, 7, 17)$\\ 
 \hline
\end{tabular}
\end{center}

An exhaustive search has ruled out the existence of a HKDF$(13)$, while 
an example of a HKDF$(19)$ has been already given.

Then, once more as in the proof of Theorem \ref{FK(q)}, we consider the case $v=q$ a nonprime prime power; we may then use the composition constructions  of Theorem \ref{composingHKDF}  to limit the search space. Having established the existence of a HK$(p)$ for $p$ prime, we can then immediately
deduce the existence of a HKDF$(p^n)$ for $p$ prime, $13\neq p\equiv1$ (mod 6); that leaves us, similarly as in the proof of Theorem \ref{FK(q)}, with the task of finding a HKDF$(p^{2})$ for $p\equiv 5,11 \pmod{12}$, for $p\in[5,\sqrt{Q(8)}]$, and a HKDF$(13^{n})$ for $n=2,3$.
For all values of $q$ as above, we are able to find an initial block of the form $B(x)=(0,1,x,x^2,x^3,x^4,x^5,x^6,x^7)$ for a suitable $x \in \F_q$ 
except for $q=p^2$ with $p\in\{5,11,13,17,23,29\}$.
For these remaining values of $q$ we found an alternative initial block $B$ as follows
\small
\begin{center}
\bgroup
\def\arraystretch{1.3}
\begin{tabular}{|c|c|}
\hline
$q$  & $B$\\
\hline
$5^2$ & $(1+3\sqrt{3},\ 0, \ 1, \ 2, \ 2+\sqrt{3}, \ t, \ 2\sqrt{3}, \ 4+3\sqrt{3}, \ 1+4\sqrt{3})$\\
\hline
$11^2$ & $(2\sqrt{-1}, \ 0, \  1, \  2, \  3+\sqrt{-1}, \ 4+\sqrt{-1}, \ 1 + 4\sqrt{-1}, \ 3+3\sqrt{-1}, \ 1+2\sqrt{-1})$\\
\hline
$13^2$ & $(2+12\sqrt{2}, \ 0, \ 1, \ 2, \ 3+\sqrt{2}, \ 4+\sqrt{2}, \ 2+8\sqrt{2}, \ 9+3\sqrt{2}, \ 3+4\sqrt{2})$\\
\hline
$17^2$ & $(5+3\sqrt{3}, 0, 1, 2, 3+\sqrt{3}, \ 4+2\sqrt{3}, \ 1+\sqrt{3}, \ 2\sqrt{3}, \ 1+3\sqrt{3})$\\
\hline
$23^2$ & $(1+11\sqrt{-1},\  0, \ 1, \ 2, \ 3+\sqrt{-1}, \ 4+\sqrt{-1}, \ 2\sqrt{-1}, \ 1+3\sqrt{-1}, \ 1+2\sqrt{-1})$\\ 
\hline
$29^2$ & $(4\sqrt{3},\ 0, \ 1, \ 2, \ 2+\sqrt{3}, \ 4+2\sqrt{3}, \ 4+3\sqrt{3}, \ 3+\sqrt{3}, \ 1+4\sqrt{3})$\\ 
 \hline
\end{tabular}
\egroup
\end{center}
\normalsize
and this completes the proof. 
\end{proof}

Note that, while in Section \ref{recur} we managed to extend existence in the Fano case to many non prime power orders 
using results on pairwise balanced designs, this approach will not work in this case, since
to apply Theorem \ref{MS} we lack a $HK(13)$ and, above all, a $HK(7)$. 
As for Fano Kaleidoscopes, the methods described above do not give us a $HK(13)$, but we know once more  that such a design might exist if we drop the requirement for it to be regular; on the other hand, obviously, we cannot have a   $HK(7)$, so there is no hope of having an almost complete existence  for orders $v\equiv 1\pmod{6}$ based on  using Theorem \ref{MS}.

\medskip
We have no non-trivial examples of a HK$(v)$
with $v\equiv 3\pmod{6}$.  We may note, though,  that such a  HK$(v)$ can never be cyclic.

\begin{proposition}
A cyclic HK$(6n+3)$ never exists. 
\end{proposition}
\begin{proof}
Among the blocks of a cyclic 2-$(6n+3,3,1)$ design we have, in particular, all cosets of $\{0,2n+1,4n+2\}$ in $\Z_{6n+3}$.
These blocks are said to be {\it short} because their orbit under $\Z_{6n+3}$ have short length $2n+1$ while all other blocks
belong to a {\it full} orbit of length $6n+3$.
Let $B$ be a block of a putative cyclic HK$(6n+3)$. If $B$ has a short line, then it is clear that all the other lines of $B$ are also short.
So all lines of $B$ should be cosets of $\{0,2n+1,4n+2\}$ in $\Z_{6n+3}$. This implies that 
any two distinct lines of $B$ are disjoint which is clearly absurd. Thus all lines of any block of our HK$(6n+3)$ should be full
which is absurd anyway. Indeed this would imply that the twelve 2-$(6n+3,3,1)$ designs associated with our HK$(6n+3)$ have no short blocks.
\end{proof}

\section{Generalizations and conclusions}\label{gen}
Fano Kaleidoscopes and Hesse Kaleidoscopes may be thought of as 
the smallest instances of the following general idea:
consider a set $\cal D$ of 2-$(k,h,1)$ designs whose vertices belong to 
a given $v$-set $\cal V$; in each design let the $b$ blocks be colored with the same $b$ different colors $c_0,c_1,\dots, c_{b-1}$.
We say that $\cal D$ is a {\it Kaleidoscope Design} of order $v$ and type $(k,h,1)$, briefly a $(k,h,1)K(v)$, if, as for FKs and HKs, 
for any two distinct points $x$, $y$ of $\cal V$ and any color $c_i$ there is exactly one design of ${\cal D}$ in which the block having color
$c_i$ contains $x$ and $y$. 
The underlying structure is a ``big'' 2-$(v,k,b)$ design, so the usual admissibility conditions for the existence of designs apply. 
Many of the ideas used in building FKs and HKs can be used in theory to build Kaleidoscope Designs; in the general case, though,  the situation becomes more complicated 
in practice. For instance, our main tool in proving Theorem \ref{FK(q)} and Theorem \ref{HK(q)} is using the results on cyclotomic classes of \cite{BP}, 
which guarantee existence for all admissible prime power orders greater that a certain bound; in the general case,  the bounds one obtains becomes  almost immediately unmanageable.
In studying HKs we have seen that the bound is  $Q(8)$ which is still approachable; 
moving on to  $(13,3,1)K(v)$, we have  a bound greater than $10^{13}$.

Let us close this work with an open question concerning Fano Kaleidoscopes over a finite field. 

When we consider a FK of order $v=2^n-1$, we may identify our point-set $\cal V$ with the space $\F_2^n\setminus \{0\}$, and look at
the following question.\begin{quote}

Is it possible to find a FK on these points having the extra property that each block is the set of non-zero vectors of a 3-dimensional vector space over $\F_2$?

\end{quote} 
Note that the answer is trivially affirmative in the case $n=3$; it is enough to use the replication trick of Remark \ref{repeat}. 

This question can naturally be discussed in the setting of $q$-analogs of designs.

Recall that a 2-$(n, k, \lambda)$ design over the field $\F_q$, or 2-$(n, k, \lambda; q)$ design is a collection $\cal B$ of $k$-dimensional
subspaces of $\F_q^n$
 with the property that any $2$-dimensional subspace of $\F_q^n$ is
contained in exactly $\lambda$ members of $\cal B$; we also say that such a design is
the $q$-analog of a 2-$(n, k, \lambda)$ design.
Research on $q$-analogs of designs has attracted quite a lot of interest recently; in particular, 
the only non-trivial known 2-$(n, k, 1; q)$ designs have recently been constructed in \cite{BEOVW} for $q=2, n=13$ and $k=3$.

In the Fano Kaleidoscope setting, we are considering the case $q=2$, $k=3$ and $\lambda=7$; note that a 2-$(n, 3, 7; 2)$ design
can be viewed as a 2-$(2^n-1,2^3-1,7)$ design in the classical sense, where the points are the elements of $\F_2^n\setminus\{0\}$,
with the additional property that $B \cup \{0\}$ is a
subspace of $\F_2^n$ for every block $B$.  We will denote by  FK$(2^n-1;2)$  a Fano Kaleidoscope FK$(2^n-1)$ having the extra property that each block is the set of non-zero vectors of a 3-dimensional vector space over $\F_2$.

When looking for Fano Kaleidoscopes, the number of points $v=2^n-1$ is congruent to 1 modulo 6, so $n$ must be odd.
We have the following  existence results.
\begin{theorem}
A $2$-$(n, 3, 7; 2)$ design exists for all odd values of $n$. 
\end{theorem}
This is a  result is due to Thomas (1987, \cite{T}) under the hypothesis $\gcd(n,6)=1$; recently the first author and A. Nakic \cite{Bq} proved that the result holds for all odd $n$.
Based on ideas from \cite{Bq}, we were able to construct examples of FK$(2^n-1;2)$ for $n=5$ and $7$; it would be interesting to have a general construction.

\end{document}